\documentclass[3p, authoryear, times]{elsarticle}

\usepackage{graphicx}

\usepackage{subcaption} 
\usepackage{morefloats}
\usepackage[section]{placeins}

\usepackage{amsmath, amssymb, amsthm}
\usepackage{bm}
\usepackage{xfrac}
\usepackage[T1]{fontenc}
\usepackage{times}
\usepackage{fix-cm}
\usepackage{geometry}
\usepackage{mathtools} 

\usepackage{enumitem}
\usepackage{cancel}
\usepackage{xcolor}
\usepackage{booktabs}
\usepackage{url}
\usepackage{framed}
\usepackage{listings}
\usepackage{comment}
\usepackage[nottoc]{tocbibind}

\usepackage{natbib}

\usepackage[hidelinks]{hyperref}
\usepackage{cleveref}

\newcommand{\Sd}{\mathbb{S}^{d-1}}

\newcommand{\norm}[1]{\left\| #1 \right\|}

\newcommand{\Ld}{L^2(\mathbb{S}^{d-1},\, \nu_d)}

\newcommand{\mtrx}[1]{\mathbf{#1}}
\newcommand{\mean}[1]{\bar{#1}}

\theoremstyle{plain}
\newtheorem{lemma}{Lemma}[section]

\newtheorem{theorem}{Theorem}[section]
\newtheorem{corollary}{Corollary}[section]
\newtheorem{result}{Result}[section]

\journal{Statistics and Probability Letters}

\begin{document}

\begin{frontmatter}

\title{Addressing parity blindness of data-driven Sobolev tests on the hypersphere} %
\author{Marcio Reverbel\corref{cor1}\fnref{fn1}}
\ead{marcio.reverbel@kuleuven.be}

\affiliation{organization={Universit\'e Libre de Bruxelles},
                city={Brussels},
                country={Belgium}
                }

\cortext[cor1]{Corresponding author}

\fntext[fn1]{Present address: Department of Mathematics, KU Leuven, Celestijnenlaan 200, B-02.29, 3001 Heverlee, Belgium.}

\begin{abstract}
We study the asymptotic behavior of the data-driven Sobolev test for testing uniformity on the (hyper)sphere. We show that it can be blind to certain contiguous alternatives and propose a simple modification of the test statistic. This adapted test retains consistency under fixed alternatives and achieves non-trivial asymptotic power against contiguous alternatives for which the original test fails. Simulation results support our theoretical findings.
\end{abstract}


\begin{keyword}

hypothesis testing \sep
directional statistics \sep 
local asymptotics \sep
contiguity

\end{keyword}
\end{frontmatter}

\section{Introduction}
\label{chap:intro}

Directional data, i.e., observations which lie on the surface of a (hyper)sphere, appear in applications where only the directions are relevant, such that the magnitude of each observation can be discarded. Such applications are found in many different fields, including astronomy, medicine, genetics, bioinformatics, image analysis, machine learning and text mining, among others \citep{Pewsey2021, Garcia2018}. 

Testing for uniformity on an independent and identically distributed (iid) sample of unit vectors on the unit sphere $\Sd := \{\bm x \in \mathbb{R}^d : \norm{\bm x} = \bm x^\top \bm x = 1 \}$ is a problem that has been widely discussed in the literature. It arises when one has directional data and needs to test whether these directions appear “at random” (uniform) or show some preferred orientation (anisotropy). For example, in 1767, John Michell \citep{Michell1783} was interested in how stars were distributed in the night sky and argued that there were too many pairs or groups of stars clustered together than would be possible in a uniformly distributed set of stars. The class of Sobolev tests presented by \citet{Gine1975} is the most extensive class of tests devised to test uniformity \citep[p.~12]{Pewsey2021}. It includes tests that preceded it historically, such as the Rayleigh \citep{Rayleigh1919} and Bingham \citep{Bingham1974} tests. 

The data-driven Sobolev test proposed by \cite{Jupp2008}, while consistent against all fixed alternatives, is blind to contiguous sequences of alternatives characterized by angular functions whose $k$\textsuperscript{th} derivatives vanish at zero for $k$ odd. That is, under these alternatives, its asymptotic distribution is the same as under the null. We propose a modification by placing a lower-bound on $\hat k$, and show that this results in a test with non-trivial asymptotic power against the aforementioned contiguous alternatives. Simulations confirm our theoretical findings.

This work is organized as follows: \Cref{chap:background} provides a theoretical introduction to Sobolev tests. \Cref{chap:developments} discusses the asymptotic properties of the data-driven Sobolev test, \Cref{chap:new_results} presents the adapted test, and \Cref{chap:simulations} compares the original and adapted data-driven tests with simulations. \Cref{chap:conclusion} offers a small conclusion with directions for further research.

\section{Preliminaries}
\label{chap:background}
Sobolev tests of uniformity on $\Sd = \{\bm{x} \in \mathbb{R}^d: \norm{\bm{x}} = 1\}$ are constructed from an infinite-dimensional orthonormal basis on $\Ld$, where $\nu_d$ is the uniform probability measure on $\Sd$. Given $n$ iid observations $\mtrx{X}^{(n)} = (\bm{x}_1, \ldots, \bm{x}_n)$ of a random vector $\bm X \in \Sd$, we would like to test the hypothesis that the observations are drawn from a uniform distribution: $\mathcal{H}_0: \bm X \sim \text{Unif}(\Sd)$. Sobolev tests \citep{Gine1975} reject the null hypothesis for high values of the test statistic: 

\begin{equation}
\label{eq:Sobolev_test}
S_{v_k} := \frac{1}{n} \sum_{i,j=1}^{n} \sum_{k=1}^{\infty} v_k^2 h_k (\bm{x}_i^\top \bm{x}_j), \quad \text{ with } \quad h_k(t) :=
\begin{cases} 
    2 \cos \big( k \arccos(t) \big), & \text{if } d = 2, \\
    \left( 1 + \frac{2k}{d-2} \right)\; \sum_{j = 0}^{\lfloor k/2\rfloor}(-1)^j c_{k,j}^{\frac{d-2}{2}} t^{k-2j}, & \text{if } d > 2,
\end{cases}, 
\end{equation}
where
$(v_k)_{k\in\mathbb{N}_1}$ is a real sequence satisfying $\sum_{k=1}^{\infty}v_k^2 d_{k} < \infty$,  
\begin{equation*}
    d_k = \binom{d + k - 3}{d - 2}
    +
    \binom{d + k - 2}{d - 2}, 
    \quad \text { and } \quad
    c_{k,j}^\lambda \;:=\; \frac{2^{\,k-2j}\,\Gamma\bigl(k-j+\lambda\bigr)}{\Gamma(\lambda)\,j!\,\bigl(k-2j\bigr)!}.
\end{equation*}

Under the null hypothesis, we have 
$$S_{v_k} \xrightarrow[\mathcal{H}_0]{\mathcal{D}} \sum_{k = 1}^{\infty}v_k^2\chi^2_{d_k},$$
where $\chi^2_{d_k}$ denotes the chi-square distribution with $d_k$ degrees of freedom.

The Sobolev test as defined in \Cref{eq:Sobolev_test} is a class of tests which depends on the sequence $v_k$ chosen, and includes the Rayleigh \citep{Rayleigh1919} and Bingham \citep{Bingham1974} tests, as well as the score test of uniformity in the exponential models proposed by \cite{Beran1979}. While a purely theoretical Sobolev test in which $v_k \neq 0$ for all $k > 0$ is consistent against all alternatives, practical implementations of this test -- with a finite number of non-zero elements in the sequence $v_k$ -- will be blind to specific sets of alternatives \citep{Mardia2000, Jupp1983, Gine1975}. 

The Rayleigh test of uniformity, $R_n$, is a Sobolev test with $v_k = \mathbb{I}[k = 1]$. We also have 

\begin{equation}
\label{eq:Rayleigh_test}
R_n =  d\,n\, \norm{\mean{\bm{x}}}^2,
\end{equation}
where $\mean{\bm{x}}$ denotes the sample average. The Rayleigh test coincides with the likelihood ratio test against von Mises alternatives, and is the locally most powerful invariant test against these alternatives \citep{Mardia2000}. However, since it is a test based on the sample average, which is insensitive to identical perturbations on opposing sides of the unit sphere, the Rayleigh test is blind to alternatives with any kind of antipodal symmetry. 

The Bingham test, $B_n$, is a Sobolev test with $v_k = \mathbb{I}[k = 2]$. It takes the form 
\begin{equation}
    \label{eq:Bingham_test}
    B_n = n\,\frac{d(d+2)}{2}\left[\frac{1}{n^2}\sum_{i, j = 1}^{n} (\bm{x}_i^\top \bm{x}_j)^2 - \frac{1}{d}\right] = n\, \frac{d(d+2)}{2}\left[\mathbf{tr}[\hat\Sigma] + 2\mean{\bm{x}}^\top\hat\Sigma\mean{\bm{x}} +  \norm{\mean{\bm{x}}}^4 - \frac{1}{d}\right],
\end{equation}
where $\hat \Sigma$ denotes the sample covariance matrix. It is blind to alternatives for which the expected value of the sample scatter matrix under the alternative is indistinguishable from the expected value under the null: $\mathbb{E}_{\mathcal{H}_0}[\bm X\bm X^\top] = d^{-1}\,\mtrx{I}_d$.

The following version of the Sobolev test was introduced by \cite{Beran1979}. Its test statistic, $S_K$, takes the form of a Sobolev test with $v_k = \mathbb{I}[k\leq K]$ for some $K < \infty$:

\begin{equation}
    \label{eq:quadratic_score_test}
    S_K := \frac{1}{n} \sum_{i,j=1}^{n} \sum_{k=1}^{K} h_k (\bm{x}_i^\top \bm{x}_j).
\end{equation}
This test plays an important role in the data-driven Sobolev test introduced by \cite{Jupp2008}.

\section{Data-driven Sobolev tests}
\label{chap:developments}

One of the most difficult aspects concerning the Sobolev tests lies in the choice of the sequence $v_k$, as this impacts the kind of alternatives the test cannot detect. In the case of the quadratic score test \eqref{eq:quadratic_score_test}, this translates into choosing $K$. \cite{Jupp2008} proposes a variant of the score test in which this choice is made by the data, via a penalized score statistic. Define
\begin{equation}
    \label{eq:penalized_score}  
    B_S(K) = S_K - p_K\;\log(n),
\end{equation}
where $S_K$ is the quadratic score test \eqref{eq:quadratic_score_test} and $p_K = \sum_{k=1}^K d_k$. The data-driven Sobolev test statistic, $S_{\Hat{k}}$, is given by:
\begin{equation}
    \label{eq:data-driven test}
    S_{\Hat{k}} := \frac{1}{n} \sum_{i,j=1}^{n} \sum_{k=1}^{\Hat{k}} h_k (\bm{x}_i^\top \bm{x}_j),
    \quad \quad \Hat{k} = \inf\left\{ K \in \mathbb{N}: B_S(K) = \underset{m \in \mathbb{N}} {\sup}B_S(m)\right\}
\end{equation}

The test rejects uniformity for large values of $S_{\Hat{k}}$. In practice, one would not calculate infinitely many $B_S$ statistics to choose $\Hat{k}$, but rather place an upper limit $M$ on $m$.

\cite{Jupp2008} derived important asymptotic properties for the data-driven Sobolev test: the convergence in probability under the null of $\Hat{k}$ to $1$, the asymptotic distribution of $S_{\Hat{k}}$ under the null, and the consistency of the test against all alternatives to uniformity. Together, these theorems tells us that, for large samples, $\Hat{k}$ tends to be close to $1$ under uniformity, which leads to a simple test statistic, and that $\Hat{k}$ tends to rise sufficiently to cause rejection of the null hypothesis under any (fixed) alternative distribution. See Theorems 3.1, 3.2 and 3.3 in \cite{Jupp2008} for more details.

The asymptotic behavior of Sobolev tests under the null hypothesis has been well explored and follows directly from Theorem 4.1 in \cite{Gine1975}. However, research into the asymptotic behavior of Sobolev tests under non-null distributions is much more recent, with \cite{Garcia2024} presenting important results under some of the most commonly considered alternative distributions. These distributions have densities of the form: 
\begin{equation}
\label{eq:angular_distribution}
    \bm{x} \mapsto c_{d, \kappa, g}\; g(\kappa\; \bm{\mu}^\top\bm{x}),
\end{equation}
where $\kappa > 0$ is a concentration parameter, $\bm{\mu} \in \Sd $ is a location parameter, and $g: \mathbb R \to \mathbb{R}^+$ is an \emph{angular function}, i.e, a function which only depends on the angle between between $\bm{x}$ and $\bm{\mu}$, and $c_{d, \kappa, g}$ is a normalizing constant. Distributions that fall under this category include the {von Mises-Fisher} distribution, with $g(s) = \exp{(s)}$; the {Watson} distribution, with $g(s) = \exp{(s^2)}$; other exponential distributions of the form $g(s) = \exp{(s^b)}$, with $b > 0$; and the {directional Cauchy distribution}, with $g(s) = (1 + 2s)^{-1}$ and $s = \kappa\; (1 - \bm{\mu}^\top\bm{x})$ instead. The \emph{angular} terminology is used here to emphasize the fact that these distributions are rotationally symmetric with respect to the location parameter, which makes them of particular interest. Let $P_{\kappa_n, g}$ denote a local version of the distribution in \Cref{eq:angular_distribution}, such that $\kappa = \kappa_n$ converges to zero as $n$ increases. Finally, we have the following result from \cite{Garcia2024}:
\begin{result}[{\citealp[Corollary 5.1]{Garcia2024}}]\label{cor:parity}
For \(d\ge2\), consider a sequence $v_k$ with only finitely many non‐zero terms such that \(v_k=0\) for each even (resp.\ odd) \(k\).  Let \(g\) be an angular function, \(q \;(\ge K_v :=\max\{k : v_k \neq 0\})\) times differentiable at zero, with \(g^{(k)}(0)=0\) for each odd (resp.\ even) \(k\in\{k_v,\dots,q\}\), where $k_v := \min\{k : v_k \neq 0\}$. Fix $\tau > 0$.  Then, under \(P_{\kappa_n, g}\) with \(\kappa_n=n^{-1/(2q)}\tau\),  
\[
  S_{v_k} \xrightarrow{\mathcal D}
  \sum_{k=1}^{K_v} v_k^2\,\chi^2_{d_{k}} \quad \text{as} \quad n\to\infty. 
\]
\end{result}
This Corollary states that if a Sobolev test only has non-zero $v_k$ for $k$ of one parity, and the $k$\textsuperscript{th} derivatives of $g$ vanish at zero for all $k$ of the same parity, then the test is blind, in the sense that its asymptotic distribution under the local alternative is no different than under the null.  
\section{An adapted data-driven Sobolev test}

\label{chap:new_results}
Based on the results of \Cref{chap:developments}, we investigate the non-null asymptotic behavior of the data-driven Sobolev test in the Le Cam sense. \cite{Jupp2008} demonstrated the test's consistency in a fixed setting. However, the test's behavior under sequences of alternatives which become increasingly close to the null distribution remains unexplored. Specifically, we will see that the data-driven test is blind to contiguous alternatives whose odd-numbered derivatives of the angular function vanish at zero. First, recall the definition of contiguity: For two sequences of probability measures, $Q_n$ and $P_n$, we say that $Q_n$ is contiguous with respect to $P_n$ (notation: $Q_n \lhd P_n$) if, for every sequence of measurable sets $A_n$, $P_n(A_n) \xrightarrow{n\to\infty} 0$ implies $Q_n(A_n) \xrightarrow{n\to\infty} 0.$ We can now present the following lemma: 

\begin{lemma}
\label{lemma:main_result}
Let $S_{\hat{k}}$ be the data-driven Sobolev test as defined in \Cref{eq:data-driven test}, and let $P_{\kappa_n}$ be a sequence of distributions. If $P_{\kappa_n} \lhd P_0$, then: 
\begin{equation}
    P_{\kappa_n}[\hat{k} = 1] \xrightarrow{n\to\infty} 1
\end{equation}
\end{lemma}
\begin{proof}
    This lemma is a direct application of the definition of contiguity. From Theorem 3.1 in \cite{Jupp2008}, and contiguity of $P_{\kappa_n}$:
    \begin{equation*}
        \underbrace{P_0[\hat{k} = 1] \xrightarrow{n\to\infty} 1}_{\text{Theorem 3.1 in \cite{Jupp2008}}}
        \implies P_0[\hat{k} > 1] \xrightarrow{n\to\infty} 0 
        \underbrace{\implies P_{\kappa_n}[\hat{k} > 1] \xrightarrow{n\to\infty} 0}_{ P_{\kappa_n} \lhd P_0} \implies P_{\kappa_n}[\hat{k} = 1] \xrightarrow{n\to\infty} 1. 
    \end{equation*}
\end{proof}

From \Cref{cor:parity}, this test will be blind to contiguous alternatives if the $k$\textsuperscript{th} derivatives of their angular functions vanish at zero for every $k$ odd. This is, of course, not desirable, as there exist other Sobolev tests with non-trivial asymptotic powers against these alternatives (such as the Bingham test), and it would be highly appropriate for a data-driven test to select one of these instead of the blind test. 

We propose a simple solution to counter this problem, which is to include at least one $k$ of each parity in the Sobolev test. Thus, an adapted data-driven Sobolev test takes the following form:   

\begin{equation}
    \label{eq:data-driven test_alternative}
    S_{\Hat{k}^*} := \frac{1}{n} \sum_{i,j=1}^{n} \sum_{k=1}^{\Hat{k}^*} h_k (\bm{x}_i^\top \bm{x}_j), \quad \quad \quad    \Hat{k}^* = \max\{\Hat{k}, 2\},
\end{equation}
and $\Hat k$ is as defined in \Cref{eq:data-driven test}.

The adapted data-driven Sobolev test $S_{\Hat{k}^*}$ is identical to the original data-driven test, except that the smallest value accepted of $\Hat{k}^*$ is 2. The test rejects uniformity for large values of $S_{\Hat{k}}$, and it is easy to see that it is still consistent against all alternatives to uniformity in a fixed setting. However, now it will not be blind to contiguous alternatives whose odd-ordered derivatives vanish at zero. The price to pay is the decrease in power when testing against alternatives that could be easily rejected at $\Hat{k} = 1$. In fact, we have the following results: 

\begin{theorem}
Let $\hat k^*$ be defined as in \Cref{eq:data-driven test_alternative}. Then
    \label{thm:k_hat_alternative}
    \begin{equation*}    
    \Hat{k}^* \overset{P}{\underset{\mathcal{H}_0}{\longrightarrow}} 2 \quad \text{ as }\; n \to \infty.
    \end{equation*}
\end{theorem}
\begin{proof}
The proof is trivial. From Theorem 3.1 in \cite{Jupp2008}, we have that $\hat k$ converges in probability to $1$ under $\mathcal H_0$. Therefore, $\hat k^*$ converges in probability to $\max\{1, 2\}$ = 2 under $\mathcal H_0$.
\end{proof}

\begin{theorem}
Let $S_{\hat k^*}$ be defined as in \Cref{eq:data-driven test_alternative}. Then 
    \label{thm:k_hat_alternative_asymptotic_distribution}
    \begin{equation*}
        S_{\Hat{k}^*} \overset{\mathcal{D}}{\underset{\mathcal{H}_0}{\longrightarrow}} \chi_{d_1+d_2}^2 \quad \text{ as }\; n \to \infty. 
    \end{equation*}
\end{theorem}
\begin{theorem}
    \label{thm:adapted_k_hat_consistency}
The test $\phi^*(\bm x^{(n)})$ which rejects uniformity for large values of $S_{\Hat{k}^*}$ is consistent against all (fixed) alternatives to uniformity, i.e:  
\begin{equation*}
\mathbb{E}_{\mathcal{H}_1}[\phi^*(\bm x^{(n)})] \xrightarrow{n\to\infty} 1
\end{equation*}
\end{theorem}

The proofs are simple analogous versions of those of Theorems 3.2 and 3.3 in \cite{Jupp2008}. In addition, the following result shows the existence of non-trivial asymptotic power for the adapted data-driven Sobolev test against any contiguous sequence of alternatives of the form in \Cref{eq:angular_distribution} and for which $g^{(k)}(0) \neq 0$ for some $k$:

\begin{corollary}
    Let $S_{\Hat{k}^*}^{(n)}$ be the adapted data-driven Sobolev test defined in \eqref{eq:data-driven test_alternative}. Let \(g\) be an angular function that is $q \;(\ge \max\{k : v_k \neq 0\})$ times differentiable at zero such that $g^{(k)}(0) \neq 0$ for at least one $k \leq q$. Let $k_*$ be the smallest $k$ for which $g^{(k)}(0) \neq 0$. Set a non-centrality parameter
    \[
    \xi_{k,k_*}(\tau)
    \;=\;
    \frac{1}{(k_*!)^2}\,w_{k}^2\,\bigl(g^{(k_*)}(0)\bigr)^2\,\tau^{2k_*}
    \left(\sum_{j=0}^{\lfloor k/2\rfloor}(-1)^j\,c^{(d-2)/2}_{k,j}\,a_{\,k+k_*-2j}\right)^2,
    \]
    where
    \begin{equation*}
    w_k :=
\begin{cases}
\sqrt{2}, & \text{ if } d = 2,\\[6pt]
\dfrac{1 + \frac{2k}{(d-2)}}{\sqrt{d_{k}}}, & \text{ if } d \ge 3
\end{cases}, \quad \quad 
a_m 
 = \mathbb{I}\Bigl[\tfrac m2\in\mathbb{N}\Bigr]\,
\prod_{r=0}^{\tfrac m2-1}\frac{1 + 2r}{d + 2r},
\end{equation*}
\normalsize
and with the convention for $a_m$ that an empty product is equal to one. Then, under \(P_{\kappa_n,g} \lhd P_{0}\)  with \(\kappa_n=n^{-1/(2k_*)}\tau\),
    \[
      S_{\Hat{k}^*}^{(n)} \xrightarrow{\mathcal D}
      \sum_{k=1}^{2} \chi^2_{d_{k}}\!\Bigl(\mathbb I[k\sim k_*]\,\xi_{k,k_*}(\tau)\Bigr),
    \]
where the relation $a \sim b$ is satisfied when $a$ and $b$ share the same parity.
\end{corollary}
\begin{proof}
We have $v_k = \mathbb{I}[k \in \{1, 2\}]$ for the adapted data-driven Sobolev test under contiguous alternatives, and $k_*$ will always exist under any of the alternative distributions considered. Direct application of Theorem 5.2 from \cite{Garcia2024} yields the final result. 
\end{proof}

\section{Simulations}
\label{chap:simulations}
In the following, we provide some simulation results comparing the original data-driven Sobolev test \citep{Jupp2008} with our adapted version. For every combination of $n \in \{200, 500, 1500\}$, $\ell \in \{2, 4, 6\}$ and $\tau \in \{0, 0.5, \ldots, 6\}$, we generated $M = 5,000$ independent random samples in $\mathbb S^2$ of von Mises-Fisher ($g(s) = \exp(s)$) and Watson distributions ($g(s) = \exp(s^2)$) with $\kappa_n = n^{-1/\ell}\tau$ (see \Cref{eq:angular_distribution}). In each of these samples, we computed the data-driven Sobolev test ($S_{\Hat{k}}$) and adapted data-driven Sobolev test ($S_{\Hat{k}^*}$). The results are presented in Figures 1a, 1b for the von Mises-Fisher alternatives, and Figures 1c, 1d for the Watson alternatives.

\begin{figure}[!ht]
    \centering
    \includegraphics[width=0.75\linewidth]{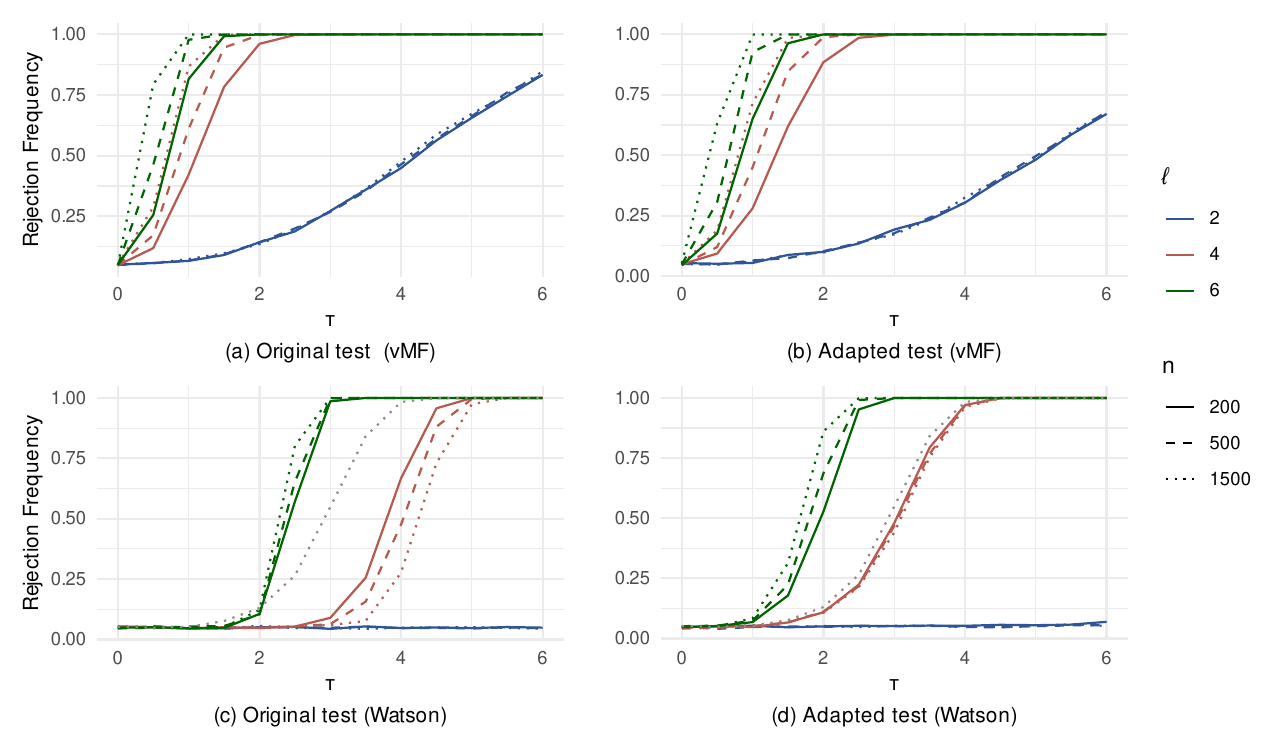}
    \caption{Rejection frequencies at asymptotic level $\alpha=5\%$ for $5000$ samples under (top) von Mises–Fisher and (bottom) Watson alternatives; left: data-driven Sobolev test, right: adapted test. Their respective angular functions are of the form $g(\kappa_n \bm x^\top \bm \mu)$, where $\kappa_n = n^{-\sfrac1\ell}\tau$ and $\bm \mu = \bm e_1$, from the standard basis. In the bottom figures, we included in gray the asymptotic power of the Bingham test, for comparison.}
    \label{fig:simulations}
\end{figure}
The Rayleigh test is the locally most powerful invariant test against von Mises-Fisher alternatives. Since these alternatives are contiguous to the null distribution when $\ell = 2$, we have $\Hat{k} \to 1$ and $\Hat{k}^* \to 2$ in probability. Therefore, Jupp's data-driven Sobolev test is the Rayleigh test in this scenario, while the adapted test is a combination of Rayleigh and Bingham tests. The Rayleigh test being the most powerful test suggests that Jupp's test should have higher asymptotic power than our adapted version. This is clearly supported by Figures 1a, 1b. The brown curves, representing the empirical rejection frequencies for $\ell = 2$, approximate the non-trivial asymptotic powers against contiguous alternatives for different values of $\tau$. When testing against these kinds of alternatives, the proposed test performs worse. However, it remains consistent against {von Mises-Fisher alternatives} for which the convergence rate $\kappa_n$ is slower than in the contiguous case ($\ell > 2$).

When testing against the Watson distribution, our theoretical results showed that Jupp's test should be blind against the contiguous alternatives, while our adapted version should show non-trivial asymptotic powers. Once again, this is clearly supported by Figures 1c, 1d. It is clear from Figure 1c that Jupp's test is blind to the contiguous alternative ($\ell = 4)$, since the increase in sample size decreases the rejection frequency for fixed $\tau$. The adapted test (Figure 1d) achieves non-trivial asymptotic powers when $\ell = 4$. For comparison, we include in gray the asymptotic power of the Bingham test ($v_k = \mathbb{I}[k = 2]$) against the contiguous Watson alternative. This indicates the cost that comes from using the data-driven tests instead of the Bingham test against this specific alternative.

\section{Conclusion}
\label{chap:conclusion}

In this work, we analyzed the data-driven Sobolev test under contiguous sequences of alternatives. This led to an adapted version of the test, correcting for the fact that the original data-driven test is blind under some contiguous alternatives. This comes at the cost of lower power against other alternatives. Simulations confirm our theoretical results. Directions for future research include exploring the asymptotic behavior of Sobolev tests under other groups of alternatives, such as multi-spiked alternatives like the Bingham distribution \citep{Garcia2024}, mixtures of distributions, as well as exploring other versions of the data-driven test.  

\section*{Acknowledgements}
The author gratefully acknowledges Thomas Verdebout (Université Libre de Bruxelles) for his guidance during the author's master's thesis which originated this work, and Johan Segers (KU Leuven) for valuable comments and careful proofreading of the manuscript.

\bibliographystyle{apalike}  
\bibliography{Bibliography}

\end{document}